\documentclass[12pt,reqno]{amsart}
\usepackage{amscd,amsmath,amsthm,amssymb}
\usepackage{color}
\usepackage{tikz}
\usepackage{url}
\usepackage{circuitikz}

\newtheorem{Theorem}{Theorem}[section]

\newtheorem{Corollary}[Theorem]{Corollary}
\newtheorem{Proposition}[Theorem]{Proposition}
\newtheorem{Remark}[Theorem]{Remark}

\def\qed{\ifhmode\textqed\fi
	\ifmmode\ifinner\quad\qedsymbol\else\dispqed\fi\fi}
\def\textqed{\unskip\nobreak\penalty50
	\hskip2em\hbox{}\nobreak\hfill\qedsymbol
	\parfillskip=0pt \finalhyphendemerits=0}
\def\dispqed{\rlap{\qquad\qedsymbol}}

\def\m{\mathfrak{m}}

\def\depth{\textup{depth\,}}

\def\im{\textup{im}}
\def\pd{\textup{proj\,dim}}

\def\reg{\textup{reg}}
\def\comp{\textup{comp}}
\def\del{\textup{del}}
\def\lk{\textup{lk}}

\newcommand{\precdot}{\prec\mathrel{\mkern-3mu}\mathrel{\cdot}}

\begin{document}
	
	\title{Algebraic study on permutation graphs}
	\author{Antonino Ficarra, Somayeh Moradi}

	\address{Antonino Ficarra, Departamento de Matem\'{a}tica, Escola de Ci\^{e}ncias e Tecnologia, Centro de Investiga\c{c}\~{a}o, Matem\'{a}tica e Aplica\c{c}\~{o}es, Instituto de Investiga\c{c}\~{a}o e Forma\c{c}\~{a}o Avan\c{c}ada, Universidade de \'{E}vora, Rua Rom\~{a}o Ramalho, 59, P--7000--671 \'{E}vora, Portugal}
	\email{antonino.ficarra@uevora.pt}\email{antficarra@unime.it}
	
	\address{Somayeh Moradi, Department of Mathematics, Faculty of Science, Ilam University, P.O.Box 69315-516, Ilam, Iran}
	\email{so.moradi@ilam.ac.ir}
	
	
	\begin{abstract}
		Let $G$ be a permutation graph. We show that $G$ is Cohen-Macaulay if and only if $G$ is unmixed and vertex decomposable. When this is the case, we obtain a combinatorial description for the $a$-invariant of $G$. Moreover, we characterize the Gorenstein permutation graphs.  
		
	\end{abstract}
	
	\maketitle\vspace*{-1.1em}
	\section*{Introduction}
	
	Permutation graphs arise naturally in combinatorics and graph theory. They are characterized as those graphs which are both comparability and co-comparability graphs of posets~\cite{PLE}, making them an intriguing subject of algebraic and combinatorial investigation. For a finite simple graph $G$ on $n$ vertices and with the edge set $E(G)$, the {\em edge ideal} of $G$, introduced by Villarreal~\cite{Vi}, is the ideal of the polynomial ring $S=K[x_i:i\in V(G)]$ over a field $K$ defined as $I(G)=(x_ix_j: \{i,j\}\in E(G))$. The main theme in the study of edge ideals is to translate the algebraic properties of the ring $S/I(G)$ to the combinatorics of the underlying graph $G$ and vice versa. The study of the Cohen-Macaulay  property of graphs has been well-established for various classes of graphs such as bipartite graphs~\cite{HH05}, very well-covered graphs~\cite{CRT,MMCRTY}, chordal graphs~\cite{HHZ},  Cameron-Walker graphs~\cite{HHKO}, fully-whiskered graphs~\cite{CN} and graphs of girth at least five~\cite{BC}. For the aforementioned classes of graphs it is shown that $G$ is Cohen-Macaulay if and only if $G$ is  unmixed and vertex decomposable.  
	
	Vertex decomposable simplicial complexes were introduced by Provan and Billera in~\cite{PB}. Their recursive definition allows to determine algebraic invariants of their Stanley-Reisner rings inductively, see~\cite{MK}. A graph $G$ is called {\em vertex decomposable} if the independence complex of $G$ is vertex decomposable. These graphs were first considered by Dochtermann-Engstr\"om~\cite{DE} and Woodroofe~\cite{W1}. 
	Any unmixed, vertex decomposable graph is Cohen-Macaulay, but the converse does not hold in general.
	
	In this work, we study the Cohen-Macaulay and the Gorenstein properties for permutation graphs. These graphs were first introduced in~\cite{EPL} and~\cite{PLE}. They form a subclass of weakly chordal graphs, as was shown in~\cite{G}. Different characterizations of permutation graphs are given in~\cite{G,L,PLE}. In Theorem~\ref{vd} we show that a permutation graph is Cohen-Macaulay if and only if it is unmixed and vertex decomposable.  
	To this aim, we use a characterization of   Cohen-Macaulay permutation graphs given in~\cite[Theorem 1.1]{CDKKV} in terms of the maximal cliques of the graph.
	This implies that the cover ideal $J(G)$ of a Cohen-Macaulay permutation graph is vertex splittable (Corollary \ref{Cor:vs}) and that the Rees algebra $\mathcal{R}(J(G))$ and the toric algebra $K[J(G)]$ are normal Cohen-Macaulay domains (Corollary \ref{Rees}).  
	
	In Theorem~\ref{Goren} we characterize the Gorenstein permutation graphs. To prove it, we use a result by Oboudi and Nikseresht~\cite{ON} regarding the Gorenstein graphs. In Proposition~\ref{a-inv} we obtain the $a$-invariant of a Cohen-Macaulay permutation graph in terms of the induced matching number and the vertex cover number of $G$. As a consequence we determine when $I(G)$ is Hilbertian, that is the Hilbert function and the Hilbert polynomial of $S/I(G)$ coincide. We conclude the paper with Proposition~\ref{biCM} which gives a combinatorial description for bi-Cohen-Macaulay graphs.
	
	\section{Preliminaries}
	
	In this section, we recall some concepts and introduce some notation which are needed in the sequel. Throughout, $G$ is a finite simple graph with the vertex set $V(G)=[n]=\{1,2,\dots,n\}$ and the edge set $E(G)$, and $S=K[x_1,\dots,x_n]$ is the polynomial ring over a field $K$. The {\em edge ideal} of $G$ is defined as the ideal of $S$, $$I(G)=(x_ix_j:\ \{i,j\}\in E(G)).$$
	
	The graph $G$ is called {\em Cohen-Macaulay}, respectively {\em Gorenstein}, if $S/I(G)$ is a Cohen-Macaulay, respectively Gorenstein ring. A subset $F\subseteq V(G)$ is called an {\em independent set} of $G$, if it contains no edge of $G$. The maximal cardinality of independent sets of $G$ is denoted by $\alpha(G)$. 
	
	A {\em vertex cover} of $G$ is a subset $C\subseteq V(G)$ which intersects all the edges of $G$ and a vertex cover which is minimal with respect to inclusion is called a {\em minimal vertex cover} of $G$. 
	The graph $G$ is called {\em unmixed} if all the minimal vertex covers of $G$ have the same cardinality.
	The {\em vertex cover number} of $G$ is defined as the minimum cardinality of the vertex covers of $G$ and is denoted by $\tau(G)$. A subset $A\subseteq V(G)$ is called a {\em clique} of $G$ if any two vertices in $A$ are adjacent in $G$. A {\em maximal clique} is a clique of $G$ which is not contained in any other clique of $G$.

	A {\em matching} of $G$ is a subset of $E(G)$ consisting of pairwise disjoint edges of $G$. The maximum size of matchings of $G$ is denoted by $\textup{m}(G)$. We say that the edges $e$ and $e'$ form a {\em gap} in $G$, if they are disjoint and no vertex in $e$ is adjacent to a vertex in $e'$. A subset $E$ of edges {\em forms a gap} in $G$, when each two elements in $E$ form a gap in $G$. The maximum cardinality of a set $E\subseteq E(G)$ which forms a gap is called the {\em induced matching number} of $G$ and is denoted by $\im(G)$.
	
	For a graph $G$, the \textit{complementary graph} of $G$ is the graph $G^c$ with the same vertex set as $G$ whose edges are the non-edges of $G$.

	For a simplicial complex $\Delta$ and a face $F\in \Delta$, the {\em link} of $F$ in
	$\Delta$ is defined as $$\lk_{\Delta}(F)=\{G\in \Delta:\ G\cap
	F=\emptyset,\, G\cup F\in \Delta\},$$ and the {\em deletion} of $F$ is the
	simplicial complex $$\del_{\Delta}(F)=\{G\in \Delta:\ G\cap
	F=\emptyset\}.$$
	
	A simplicial complex $\Delta$ is  called  {\em vertex decomposable} if either
	$\Delta$ is a simplex, or $\Delta$ contains a vertex $x$ such that
	\begin{itemize}
		\item[(i)] both $\del_{\Delta}(x)$ and $\lk_{\Delta}(x)$ are vertex decomposable, and
		\item[(ii)] any facet of $\del_{\Delta}(x)$ is a facet of $\Delta$.
	\end{itemize}
	A vertex $x$ which satisfies condition (ii) is called a {\em shedding vertex} of $\Delta$.\smallskip
	
	The {\em independence complex} of a graph $G$ is defined as the simplicial complex
	$$\Delta_G=\{F\subseteq V(G):\  F \text{ is an independent set of } G  \}.$$  
	
	The graph $G$ is called {\em vertex decomposable} if $\Delta_G$ is vertex decomposable.
	
	Vertex decomposability has a nice translation  to independence complexes of graphs. For a vertex $i\in V(G)$, let $N_G(i)$ be the set of all vertices of $G$ adjacent to $i$ and let $N_G[i]=N_G(i)\cup\{i\}$. Translating the definition of vertex decomposable to independence complexes of graphs we have that:
	
	A graph $G$ is vertex decomposable, if either $G$ consists of isolated vertices or it has a vertex $i$ such that
	
	\begin{enumerate}
		\item[(i)] $G\setminus\{i\}$ and $G\setminus N_G[i]$ are vertex decomposable.
		\item[(ii)] Any maximal independent set of $G\setminus\{i\}$ is a maximal independent set of $G$.
	\end{enumerate}
	
	It can be easily seen that (ii) is equivalent to say that no independent set of $G\setminus N_G[i]$ is a maximal independent set of $G\setminus\{i\}$. Such a vertex $i$ is called a {\em shedding vertex} of $G$.
	
	Let $\sigma=(k_1,\ldots,k_n)$ be a permutation of $[n]$ that is $\sigma(i)=k_i$ for all $i$. The \emph{permutation graph} $G(\sigma)$ corresponding to $\sigma$ is the graph on the vertex set $[n]$ for which $\{i,j\}\in E(G(\sigma))$ if and only if $i<j$ and $j$ appears before $i$ in the list $k_1,\ldots,k_n$. For example if $\sigma=(2,4,5,1,3)$, then $G(\sigma)$ is the graph on $[5]$ with the edge set $\{\{1,2\},\{1,4\},\{3,4\},\{1,5\},\{3,5\}\}$.
	
	For a poset $(P,\prec)$ with the vertex set $V=\{v_1,\dots,v_n\}$, the \textit{comparability graph} $\textup{comp}(P)$ of $P$ is defined to be the graph on $V$ with $\{v_i,v_j\}\in E(\textup{comp}(P))$ if and only if $v_i$ and $v_j$ are comparable in $P$ and $v_i\ne v_j$. Whereas, the \textit{co-comparability graph} $\textup{co-comp}(P)$ of $P$ is defined to be the graph on $V$ with $\{v_i,v_j\}\in E(\textup{co-comp}(P))$ if and only if $v_i$ and $v_j$ are not comparable in $P$. In other words, $\textup{co-comp}(P)$ is the complementary graph of $\textup{comp}(P)$. By \cite[Theorem 3]{PLE}, a graph $G$ is a permutation graph if and only if $G$ is both a comparability and a co-comparability graph.
	
	For two elements $x,y$ in a poset $(P,\prec)$, we write $x\precdot y$, whenever $x\prec y$ and if $z$ is an element in $P$ with $x\preceq z\preceq y$, then $z=x$ or $z=y$.
	
	A graph $G$ is called {\em weakly chordal} if $G$ and $G^c$ have no induced cycles of length $m\ge 5$. A cycle in $G$ is called an {\em induced cycle} if no two non-consecutive vertices in the cycle are adjacent in 
	$G$.
	
	We denote by $K_n$ the complete graph on $n$ vertices and by $P_n$ the path graph on $n$ vertices. Notice that $K_n$ and $P_n^c$ are permutation graphs.
	
	\section{Cohen-Macaulay and Gorenstein permutation graphs}
	
	Permutation graphs are characterized in terms of the existence of a so-called cohesive order on their vertex sets. 
	A graph $G$ is said to have a {\em cohesive order} if there is a labeling $[n]$ on $V(G)$ such that
	\begin{itemize}
		\item [(i)] If $i<j<k$ and $\{i,j\}\in E(G)$, $\{j,k\}\in E(G)$, then $\{i,k\}\in E(G)$. 
		\item [(ii)] If $i<j<k$ and $\{i,k\}\in E(G)$, then  $\{i,j\}\in E(G)$ or $\{j,k\}\in E(G)$. 
	\end{itemize}

	\begin{Theorem}\cite[Theorem 2.3]{GRR}\label{cohesive}
		A graph $G$ is a permutation graph if and only if it has a cohesive order.
	\end{Theorem}
	
	The following characterization of Cohen-Macaulay permutation graphs by Cheri, et al.~\cite{CDKKV} will be extensively used in the proof of Theorem~\ref{vd}. 
	
	\begin{Theorem}\cite[Theorem 1.1]{CDKKV}\label{CM} Let $G$ be a permutation graph. Then the following statements are equivalent.
		\begin{enumerate}
			\item [(i)] $G$ is Cohen-Macaulay.
			\item [(ii)] $G$ is unmixed and there exists a unique way of partitioning $V(G)$ into $r$ disjoint maximal cliques, where $r$ is the cardinality of a maximal independent set of $G$. 
		\end{enumerate}
	\end{Theorem}
	
	Now, we are in the position to prove
	
	\begin{Theorem}\label{vd}
		Let $G$ be a permutation graph. The following statements are equivalent.
		\begin{enumerate}
			\item [(i)] $G$ is Cohen-Macaulay.
			\item [(ii)] $G$ is unmixed and vertex decomposable.
			\item [(iii)] $G$ is unmixed and shellable.
		\end{enumerate}
	\end{Theorem}
	
	\begin{proof}
		(i) $\Rightarrow$  (ii): Without loss of generality we assume that $G$ has no isolated vertices. By Theorem~\ref{cohesive}, we may assume that $V(G)=[n]$ is a cohesive order of $G$ in the given labeling. 
		For $i,j\in [n]$, we set $i\prec j$, if $i<j$ and $\{i,j\}\in E(G)$. It follows from the property (i) of cohesive order that ($V(G),\prec)$ is a poset. We denote this poset by $P$. Then it is clear that $G=\comp(P)$, and the maximal cliques of $G$ are just the maximal chains in $P$.
		
		Since $G$ is Cohen-Macaulay, it is unmixed. We may assume that the cardinality of any maximal independent set of $G$ is $r$.  By Theorem~\ref{CM}, $V(G)$ is partitioned in a unique way into $r$ disjoint maximal chains in $P$. Let $V(G)=A_1\cup\cdots \cup A_r$  be this unique partition. Let $j_k=\max A_k$ and $i_k\in A_k$ be the unique element in $A_k$ with $i_k\precdot j_k$ for $1\leq k\leq r$.  
		
		\textbf{Claim}. There exists an integer $1\leq t\leq r$, such that $$\{s\in [n]:\ i_t\prec s\}=\{j_t\}.$$
		
		\textbf{Proof of the claim}. Suppose on the contrary that this is not the case. Then for any $1\leq k\leq r$, there exists $1\leq \ell_k\leq r$ with $\ell_k\neq k$ such that  $i_k\prec j_{\ell_k}$. Without loss of generality we may assume that $i_1\prec j_2$. 
		
		First suppose that $i_2\prec j_1$. We show that $B_1=(A_1\setminus \{j_1\})\cup\{j_2\}$ and $B_2=(A_2\setminus \{j_2\})\cup\{j_1\}$ are maximal chains of $P$. Since $i_1\prec j_2$ and $i_2\prec j_1$,  $B_1$ and $B_2$ are chains of $P$. Suppose that $B_1$ is not a maximal chain of $P$. This would mean that there exists $s\in [n]$ with $i_1\prec s\prec j_2$. Let $F$ be a maximal independent set of $G$ which contains $s$. Since $|F|=r$, $F$ contains precisely one element from each $A_k$. Since $i_1\prec s$ and $s\in F$, it follows that $F\cap A_1=\{j_1\}$. Now, from $j_1\in F$ and  $i_2\prec j_1$, it follows that $F\cap A_2=\{j_2\}$.
		Thus $s,j_2\in F$. This contradicts to $s\prec j_2$. Therefore, $B_1$ is a maximal chain of $P$. The same argument shows that $B_2$ is a maximal chain of $P$. Then $V(G)=B_1\cup B_2\cup A_3\cup\cdots\cup A_r$ is another partition of $V(G)$ into maximal chains of $P$ and this contradicts to Theorem~\ref{CM}.

		So we have $i_2\nprec j_1$. Then by our assumption we may assume that $i_2\prec j_3$. Similar to the argument in the previous paragraph, if $i_3\prec j_1$, then $V(G)=B_1\cup B_2\cup B_3\cup A_4\cup\cdots\cup A_r$ is another partition of $V(G)$ into maximal chains of $P$, where   $B_1=(A_1\setminus \{j_1\})\cup\{j_2\}$, $B_2=(A_2\setminus \{j_2\})\cup\{j_3\}$ and $B_3=(A_3\setminus \{j_3\})\cup\{j_1\}$, and this again contradicts to Theorem~\ref{CM}. Moreover, if $i_3\prec j_2$, then $V(G)=A_1\cup B_2\cup B_3\cup A_4\cup\cdots\cup A_r$ is another partition of $V(G)$ into maximal chains of $P$, where     $B_2=(A_2\setminus \{j_2\})\cup\{j_3\}$ and $B_3=(A_3\setminus \{j_3\})\cup\{j_2\}$, which is absurd. Hence, we may assume $i_3\prec j_4$.
		
		Proceeding with the same argument, and after relabeling we obtain that $i_k\prec j_{k+1}$ for $1\leq k\leq r-1$. Our assumption implies that $i_r\prec j_s$ for some $s<r$. We set $B_k=(A_k\setminus \{j_k\})\cup\{j_{k+1}\}$ for $s\leq k\leq r-1$ and $B_r=(A_r\setminus \{j_r\})\cup\{j_s\}$. We show that each $B_k$ is a maximal chain of $P$. To simplify the notation, we set $j_{r+1}=j_s$. Fix an integer $s\leq q\leq r$ and suppose that $B_q$ is not a maximal chain of $P$. Then $i_q\prec m\prec j_{q+1}$ for some $m$. We let $F$ be a maximal independent set of $G$ which contains $m$. Since $F$ contains precisely one element from each of $A_1,\ldots,A_r$, from $i_q\prec m$, we obtain $F\cap A_q=\{j_q\}$. Since $i_{q-1}\prec j_q$ and $j_q\in F$, we get $F\cap A_{q-1}=\{j_{q-1}\}$. Similarly, the relations $i_k\prec j_{k+1}$, imply that  $F\cap A_{k}=\{j_{k}\}$ for all $s\leq k\leq r$. Therefore, $m,j_{q+1}\in F$. This contradicts to $m\prec j_{q+1}$. Thus we have proved that each $B_k$ is a maximal chain of $P$. Having this, we obtain the partition $(\bigcup_{i=1}^{s-1} A_i)\cup (\bigcup_{i=s}^r B_i)$ of $V(G)$ into maximal chains of $P$, which is different from $\bigcup_{i=1}^r A_i$. This contradicts to Theorem~\ref{CM}. So our claim is proved.\hfill$\square$ 
		
		Let $t$ be an integer satisfying the claim. Without loss of generality we let $t=1$. We prove that $j_1$ is a shedding vertex of $G$ so that $G'=G\setminus \{j_1\}$ and $G''=G\setminus N_G[j_1]$ are vertex decomposable. This will show that $G$ is vertex decomposable.
		Since any induced subgraph of a permutation graph is a permutaion graph, $G'$ and $G''$ are permutation graphs. We show that they are Cohen-Macaulay. Then by induction on the number of vertices of the graph, it follows that they are vertex decomposable. By~\cite[Proposition 4.3]{Vi}, $G''$ is Cohen-Macaulay. 
		
		To show that $G'$ is Cohen-Macaulay, we set $P'$ to be the poset obtaining from $P$ by removing $j_1$. Then $P'=\comp(G')$.   
		By our assumption on $i_1$, the chain $A'_1=A_1\setminus \{j_1\}$ is a maximal chain of $P'$. Thus $V(G')=A'_1\cup A_2\cup\cdots\cup A_r$ is a partition of $V(G')$ into maximal chains of $P'$. We show that $V(G')$ is uniquely partitioned into maximal chains of $P'$. Let $V(G')=C_1\cup C_2\cup \cdots\cup C_p$ be an arbitrary partition of $V(G')$ into maximal chains $C_k$ of $P'$. First notice that the set consisting of the maximal elements of $C_1,\ldots,C_p$ is an independent set of $G'$ and hence an independent set of $G$ of cardinality $p$. Therefore, $p\leq r$. 
		Moreover, by the assumption on $i_1$, we know that $i_1$ is a maximal element of $P'$. So it is the maximal element of some maximal chain, let say $C_1$. Then $C_1\cup\{j_1\}$ is a maximal chain of $P$, since $i_1\precdot j_1$.
		We show that each $C_k$ for $2\leq k\leq p$ is a maximal chain of $P$. Suppose this is not the case. Then $C_h\cup\{j_1\}$ is a chain in $P$ for some $2\leq h\leq p$. Now, consider a maximal independent set $F$ of $G$ with $j_1\in F$. Since $F$ has at most one element from each chain of $P$ and $C_1\cup\{j_1\}$ and $C_h\cup\{j_1\}$ are chains in $P$, we conclude that $F\cap C_1=F\cap C_h=\emptyset$ and $|F\cap C_{\ell}|\leq 1$ for all $\ell$. Thus 
		$$r-1=|F\setminus\{j_1\}|= \sum_{\ell=1}^p|F\cap C_{\ell}|\leq p-2.$$
		
		This contradicts to $p\leq r$. So $C_2,\ldots,C_p$ are maximal chains of $P$, and hence $V(G)=(C_1\cup\{j_1\})\cup C_2\cup\cdots\cup C_p$ is a partition of $V(G)$ into maximal chains of $P$. Since such a partition is unique, we obtain $p=r$, $A_1=C_1\cup\{j_1\}$ and, after a suitable relabeling, $A_k=C_k$ for $2\leq k\leq r$. So $V(G')$ is uniquely partitioned into maximal chains as $V(G')=A'_1\cup A_2\cup\cdots\cup A_r$. 
		
		Next, we show that any maximal independent set of $G'$ is a maximal independent set of $G$. This will show that $j_1$ is a shedding vertex of $G$ and that $G'$ is unmixed. Consider a maximal independent set $F$ of $G'$.  If $i_1\in F$, then  $F\cup\{j_1\}$ is not an independent set of $G$. In other words, $F$ is a maximal independent set of $G$.
		
		Now, assume that $i_1\notin F$. Since $F$ is maximal, this means that $b\in F$ for some $b\in N_{G'}(i_1)$. Otherwise, $F\cup\{i_1\}$ would be an independent set of $G'$ which strictly contains $F$. By our assumption on $i_1$, we have $N_{G}[i_1]\subseteq N_{G}[j_1]$. Indeed, it follows from the equality $\{s\in [n]:\ i_1\prec s\}=\{j_1\}$ that  if $\{s,i_1\}\in E(G)$ with $s\neq j_1$, then we have $s\prec i_1$. So $s\prec i_1\prec j_1$, and hence $\{s,j_1\}\in E(G)$. So $N_{G}[i_1]\subseteq N_{G}[j_1]$.  Therefore, $b\in N_{G}[j_1]$. Since $b\neq j_1$ we have $b\in N_{G}(j_1)\cap F$.  This shows that $F\cup\{j_1\}$ is not an independent set of $G$. So $F$ is a maximal independent set of $G$. 
		This implies that $G'$ is unmixed and $j_1$ is a shedding vertex of $G$.
		Now, by Theorem~\ref{CM}, we conclude that $G'$ is Cohen-Macaulay. So by induction, $G'$ is vertex decomposable. The proof is complete.   
		
		(ii) $\Rightarrow$  (iii) and 	(iii) $\Rightarrow$  (i) follow from~\cite[Theorem 11.3]{BW} and~\cite[Theorem 8.2.6]{HH}, respectively. 
	\end{proof}

	\begin{Remark}
		A permutation graph is not necessarily vertex decomposable. Indeed, the graph $G$ depicted below is an unmixed permutation graph, which is not vertex decomposable. It follows from Theorem~\ref{vd} that $G$ is not Cohen-Macaulay. 
		
		\medskip
		
		
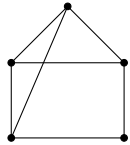
\begin{center}
			\begin{tikzpicture}[scale=0.25]\label{Pic1}
				\draw[-] (6.,6.)--(9.,3.);
				\draw[-] (6.,6.)--(3.,3.);
				\draw[-] (9.,3.)--(9.,-1.);
				\draw[-] (3.,-1.)--(3.,3.);
				\draw[-] (9.,-1.)--(3.,-1.);
				\draw[-] (9.,3.)--(3.,3.);
				\draw[-] (6.,6.)--(3.,-1.);
				\filldraw[black] (6,6) circle (5pt); 
				
				\filldraw (9,3) circle (5pt); 
				\filldraw (9,-1) circle (5pt); 
				\filldraw (3,-1) circle (5pt); 
				\filldraw (3,3) circle (5pt); 
			\end{tikzpicture}
		\end{center}
	\end{Remark}
	
	Vertex splittable ideals were defined in~\cite{MK}. They appear as the Alexander duals of the Stanley-Reisner ideals of vertex decomposable simplicial complexes. A monomial ideal $I\subset S$ is called  {\em vertex splittable} if it can be obtained by the following recursive procedure.
	\begin{itemize}
		\item[(i)] If $u$ is a monomial and $I=(u)$, $I=(0)$ or $I=S$, then $I$ is vertex splittable.
		\item[(ii)] If there is a variable $x_i$ and vertex splittable ideals $I_1$ and $I_2$ of $K[X\setminus\{{x}_{i}\}]$ so that $I=x_iI_1+I_2$, $I_2\subseteq I_1$ and $\mathcal{G}(I)$ is the disjoint union of $\mathcal{G}(x_iI_1)$ and $\mathcal{G}(I_2)$, then $I$ is vertex splittable.
	\end{itemize}
	
	The {\em cover ideal} $J(G)$ of a graph $G$ is defined as the monomial ideal generated by those monomials whose support is a vertex cover of $G$.
	
	The following corollary in obtained from Theorem~\ref{vd} and \cite[Theorem 2.2]{MK}.

	\begin{Corollary}\label{Cor:vs}
		Let $G$ be an unmixed permutation graph, and let $J(G)$ be the cover ideal of $G$. 
		The following are equivalent:
		
		\begin{enumerate}
			\item [(i)]  $J(G)$ is vertex splittable.
			\item [(ii)] $J(G)$ has linear quotients.
			\item [(iii)] $G$ is Cohen-Macaulay.
		\end{enumerate}
	\end{Corollary}
	
	We expect that for a Cohen-Macaulay permutation graph $G$, all powers of the cover ideal $J(G)$ have linear resolution.

	\begin{Corollary}\label{Rees}
		Let $G$ be a Cohen-Macaulay permutation graph. Then
		\begin{enumerate}
			\item [(a)] The Rees algebra $\mathcal{R}(J(G))$ and the toric algebra $K[J(G)]$ are normal Cohen-Macaulay domains 
			\item [(b)]  $J(G)$ satisfies the strong persistence property.
			\item [(c)] $\lim_{k\to\infty} \depth S/J(G)^k=n-\ell(J(G))$.
			\item [(d)] $\reg(K[I(G)])\leq \textup{m}(G)$.
			\item [(e)] $\textup{m}(G)\leq \reg(\mathcal{R}(I(G)))\leq \textup{m}(G)+1$. 
			
		\end{enumerate}
	\end{Corollary}
	
	\begin{proof}
		(a), (b) and (c) follow from Theorem~\ref{vd} and \cite[Theorem 3.1]{M}. Whereas, (d) and (e) follow from Theorem~\ref{vd}, \cite[Theorem 1]{HH20a} and \cite[Theorem 2.2]{HH20b}.
	\end{proof}
	
	Let $(R,\m,K)$ be either a local ring or a standard graded $K$-algebra, with (graded) maximal ideal $\m$, which is Cohen-Macaulay and admits a canonical module $\omega_R$. The \textit{canonical trace} of $R$ is defined as the ideal
	$$
	\textup{tr}(\omega_R)=\sum_{\varphi\in\textup{Hom}_R(\omega_R,R)}\varphi(\omega_R).
	$$
	
	Following \cite{HHS}, we say that $R$ is \textit{nearly Gorenstein} if $\m\subseteq\textup{tr}(\omega_R)$. It is clear from the definition that any Gorenstein ring is nearly Gorenstein.
	
	We say that a graph $G$ is \textit{nearly Gorenstein} if $S/I(G)$ is a nearly Gorenstein ring.
	
	Next we characterize Gorenstein and nearly Gorenstein permutation graphs. To this aim we use the properties of the poset $P$ associated to a Cohen-Macaulay permutation  graph $G$, employed in the proof of Theorem~\ref{vd}. 
	
	\begin{Theorem}\label{Goren}
		Let $G$ be a permutation graph without isolated vertices. Then 
		\begin{enumerate}
			\item [(a)] $G$ is Gorenstein if and only if $G$ is the disjoint union of edges.
			\item [(b)] $G$ is nearly Gorenstein but not Gorenstein if and only if   $G$ is either $K_n$ or $P_n^c$ for some $n\geq 3$.  
		\end{enumerate}   
	\end{Theorem}
	\begin{proof}
		(a) If $G$ is the disjoint union of edges, then $I(G)$ is a complete intersection, and so $G$ is Gorenstein. Conversely, suppose that $G$ is Gorenstein. Since $G$ is   Cohen-Macaulay, we will adopt the notation and results shown in the proof of Theorem~\ref{vd}. The set $L=\{j_1,\ldots,j_r\}$ is a maximal independent set of $G$. Since $G$ is unmixed, this means that $\alpha(G)=r$. If $\alpha(G)=1$, then $P$ is a chain, which means that $G$ is a complete graph. On the other hand, the only complete graph  which in Gorenstein is $K_2$. Hence $G$ is just an edge. Now, let $\alpha(G)\geq 2$. Set $e_{\ell}=\{i_{\ell},j_{\ell}\}$ for $1\leq \ell\leq r$.
		We show that $G$ is the disjoint union of the edges  $e_1,\ldots,e_r$.
		
		For any $F\subseteq [n]$, we set $G_F=G\setminus N_G[F]$. Then by~\cite[Theorem 2.3]{ON},
		for any independent set $F$ of $G$ with $|F|=r-2$, we have $G_F=C_m^c$, where $C_m$ denotes the cycle graph on $m\geq 4$ vertices. Since any permutation graph is weakly chordal, we obtain $m=4$. Consider a subset $F\subset L$ with $|F|=r-2$. Without loss of generality assume that $F=L\setminus\{j_1,j_2\}$. Since $G_F=C_4^c$, we get $E(G_F)=\{e_1,e_2\}$. This means that $e_1$ and $e_2$ form a gap in $G$. 
		By choosing the set $F$ as $L\setminus\{j_s,j_t\}$ for any $s\neq t$, with the same argument we conclude that any two edges $e_s,e_t$ form a gap. This means that $e_1,\ldots,e_r$ form a gap in $G$. If $|V(G)|=2r$, it follows that $E(G)=\{e_1,\ldots,e_r\}$, as desired. 
		
		Now, by contradiction assume that $|V(G)|>2r$. Since $G$ has no isolated vertices, we have $|A_{\ell}|\geq 2$ for $1\leq\ell\leq r$. Then by the assumption that $|V(G)|>2r$, we may assume that $|A_{\ell}|>2$ for $1\leq\ell\leq s$ and $|A_{\ell}|=2$ for $ \ell> s$, where $s$ is an integer with $1\leq s\leq r$. For any maximal chain $A_{\ell}$ with $|A_{\ell}|>2$, let $t_{\ell}\in A_{\ell}$ be the element with $t_{\ell} \precdot i_{\ell}$. 
		Since $e_1,\ldots,e_r$ form a gap, the set  
		$L'=\{i_1,\ldots,i_r\}$ is an independent set of $G$. For any $1\leq\ell\leq s$, choose $F_{\ell}\subset L'$ such that $|F|=r-2$ and $i_{\ell}\notin F$. Then $G_{F_{\ell}}=C_4^c$. Moreover, $e_\ell=\{i_\ell,j_\ell\}$ is and edge of $G_{F_{\ell}}$. This together with $t_\ell\in A_\ell$ implies that $t_\ell\in N_G(F_\ell)$. In other words, for any $1\leq\ell\leq s$ there exists $h_\ell\neq \ell$ such that $\{t_\ell,i_{h_\ell}\}\in E(G)$. This in fact means that for any $1\leq\ell\leq s$,  $t_\ell \prec i_{h_\ell}$. Hence, $\{i_{h_\ell},i_\ell\}\subseteq \{k:\ t_\ell\prec k\}$. We show that this is not possible.
		
		We set $G_1=G\setminus \{j_1\}$ and $G_p=G_{p-1}\setminus \{j_{p}\}$ for $2\leq p\leq r$. Since $e_1,\ldots,e_r$ form a gap in $G$, for any integer $p$ we have
		$\{k\in [n]:\ i_p\prec k\}=\{j_p\}$. Then as is shown in the proof of Theorem~\ref{vd}, each $j_p$ is a shedding vertex of $G_p$ and each $G_p$ is a Cohen-Macaulay permutation graph. In particular $G_r=G\setminus \{j_1,\ldots,j_r\}$ is Cohen-Macaulay. Let $P'=P\setminus \{j_1,\ldots,j_r\}$. Then clearly $G_r=\comp(P')$ and $P'$ is the disjoint union of the maximal chains $A'_\ell=A_\ell\setminus\{j_\ell\}$. Moreover, $i_\ell$ is the maximal element of $A'_\ell$ for all $\ell$. As the proof of Theorem~\ref{vd}(Claim) shows, there should exist an integer $1\leq \ell\leq s$ such that $\{k\in V(P'):\ t_\ell\prec k\}=\{i_\ell\}$. This contradicts to $\{i_{h_\ell},i_\ell\}\subseteq \{k:\ t_\ell\prec k\}$ for all $1\leq\ell\leq s$. So we have $|V(G)|=2r$, and this concludes the proof of (a).
		
		(b) By~\cite[Theorem A(Y)]{MV}, $G$ is nearly Gorenstein but not Gorenstein if and only if $\Delta_G$ is  isomorphic either to the disjoin union of $n$ vertices or to a path on $n$ vertices. This implies that either $G=K_n$ or $G=P_n^c$. Since both $K_n$ and $P_n^c$ are Cohen-Macaulay permutation graphs, the result follows.
	\end{proof}
	
	In the next proposition we give a combinatorial description for the $a$-invariant $a(S/I(G))$ of the ring $S/I(G)$, when $G$ is a Cohen-Macaulay permutation graph. A graded ideal $I\subset S$ is called {\em Hilbertian} if $P_{S/I}(t)=H(S/I,t)$ for all $t\geq 0$, where $P_{S/I}(t)$ and $H(S/I,t)$ denote the Hilbert polynomial and the Hilbert function of $S/I$, respectively.
	
	\begin{Proposition}\label{a-inv}
		Let $G$ be a Cohen-Macaulay permutation graph on $n$ vertices. Then
		\begin{enumerate}
			\item [(a)]  $\reg(S/I(G))=\im(G)$.
			\item [(b)] $a(S/I(G))=\im(G)+\tau(G)-n$.
			\item [(c)] $I(G)$ is Hilbertian if and only if $\tau(G)+\im(G)<n$.  
		\end{enumerate}
	\end{Proposition}
	
	\begin{proof}
		(a) follows from the fact  that any permutation graph is weakly chordal, together with \cite[Theorem 14]{W}, which shows that if $G$ is a weakly chordal graph, then $\reg(S/I(G))=\im(G)$. 
		
		(b)   
		Since $G$ is Cohen-Macaulay, we know that the degree of the $h$-polynomial $h(t)$ in the Hilbert series of $S/I(G)$ is equal to $\reg(S/I(G))$, see ~\cite[Corollary B.28]{V}. Hence, 
		$a(S/I(G))=\reg(S/I(G))-d$, where $d=\dim(S/I(G))$.
		Using (a) and the equality $d=n-\tau(G)$, we obtain $a(S/I(G))=\im(G)+\tau(G)-n$.

		(c) By \cite[Lemma 5.3]{NRV}, the ideal $I(G)$ is Hilbertian if and only if $a(S/I(G))<0$. By (ii) we have $a(S/I(G))<0$ if and only if $\tau(G)+\im(G)<n$.
	\end{proof}
	
	A graph $G$ is called {\em bi-Cohen-Macaulay} if $S/I(G)$ and $S/J(G)$ are Cohen-Macaulay rings.
	Combining Theorem~\ref{vd} and Proposition~\ref{a-inv}(a) we obtain
	
	\begin{Proposition}\label{biCM}
		Let $G$ be a permutation graph on $n$ vertices. Then
		$G$ is  bi-Cohen-Macaulay if and only if $G$ is an unmixed vertex decomposable graph and $\im(G)=1$.  
	\end{Proposition}
	
	\begin{proof}
		Having Theorem~\ref{vd}, it is enough to show that $S/J(G)$ is Cohen-Macaulay if and only if $\im(G)=1$.
		By~\cite[Theorem 2.1]{T} and Proposition~\ref{a-inv}(a) we have $\pd(S/J(G))=\reg(I(G))=\im(G)+1$. So $\depth(S/J(G))=n-\im(G)-1$. Moreover, since $J(G)$ is unmixed of height two, we have $\dim(S/J(G))=n-2$. So $S/J(G)$ is Cohen-Macaulay if and only if $\im(G)=1$. 
	\end{proof}
	
	\bigskip

	\noindent\textbf{Acknowledgment.}
	A. Ficarra was partly supported by INDAM (Istituto Nazionale di Alta Matematica), and also by the Grant JDC2023-051705-I funded by
	MICIU/AEI/10.13039/501100011033 and by the FSE+. S. Moradi is supported by the Alexander von Humboldt Foundation.

\end{document}